\documentclass[12pt]{amsart} 
\usepackage[ansinew]{inputenc}
\usepackage{fancyhdr}
\usepackage{epsfig}
\usepackage{epic}
\usepackage{eepic}
\usepackage{amsfonts}
\usepackage{threeparttable}
\usepackage{amscd}
\usepackage{here}
\usepackage{graphicx}
\usepackage{lscape}
\usepackage{hyperref}
\usepackage{xcolor}
\usepackage{tabularx}
\usepackage{subfigure}
\usepackage{longtable}
\usepackage{amsmath,amsthm}
\usepackage{amssymb,latexsym}
\usepackage{enumerate}
\usepackage{mathrsfs}
\usepackage{layout}
\usepackage{rotating} 

\newtheorem{lemma}{Lemma} 
\newtheorem{teor}{\sc {\textbf {Theorem}}}
\newtheorem{defin}{\sc {\textbf {Definition}}}
\newtheorem{coro}{\sc {\textbf {Corollary}}}

\newcommand{\Dr}{\mathbb{R}}
\newcommand{\F}{\mathcal{F}}
\newcommand{\R}{\mathcal{R}}
\newcommand{\Df}{\mathsf{F}}
\newcommand{\Or}{\mathcal{O}}

\newcommand{\Dz}{\mathbb{Z}}

\title{Shadowing for codimension one sectional-Anosov flows}
\author[Arbieto, A., L\'opez, A.M., S\'anchez, Y.]{Arbieto, A., L\'opez, A.M., S\'anchez, Y.}

\address{Arbieto, A. \\
Instituto de Matem\'{a}tica \\ 
Universidad Federal do Rio de Janeiro, Rio de Janeiro, Brasil} 
\email{	alexander.arbieto@gmail.com}
\address{L\'{o}pez, A.M. \\
Departamento de Matem\'{a}tica \\ 
Universidad Federal Rural do Rio de Janeiro, Rio de Janeiro, Brasil} 
\email{andresmlopezb@gmail.com, barragan@pg.im.ufrj.br}
\address{S\'{a}nchez, Y. \\
Departamento de Matem\'{a}ticas \\ 
Universidad Nacional de Colombia, Bogot\'{a}, Colombia} 
\email{yasanchezr@unal.edu.co}

\date{\today} 


\begin{document} 

\begin{abstract}
In hyperbolic dynamics, a well-known result is that every hyperbolic attracting set, have a finete pseudo-orbit tracing property (FPOTP). It's natural to wonder if this result is maintained in the sectional-hyperbolic dynamics; Komuro in \cite{komu}, provides a negative answer for this question, proving that the geometric Lorenz Attractor doesn't have a  FPOTP. In this paper, we generalized the result of Komuro, we prove that every codimension one sectional-hyperbolic attractor set with a unique singularity Lorenz-like, which is of boundary-type, does not have FPOTP.

\end{abstract}

\maketitle 

\section{Introduction}

In the dynamic hyperbolic theory, we have Anosov diffeomorphisms have the {\em shadowing property or  pseudo-orbit tracing property (POTP)} and furthermore that any pseudo-orbit has a unique point shadowing it. Thus, the shadowing property turn very useful, where this one subsequently appears in several hyperbolic systems, Markov partitions, symbolic dynamic between others.[referencia clarck robinson]. On the other hand, the Anosov flows also  satisfy the shadowing property. \\

In Komuro, variations of shadowing appeared in the literature driven by different approaches in itself, where he finds some relations between them, to solve the property existence and using it.\\

The sectional-hyperbolic sets are a more general class than hyperbolic sets, since they include these and other non-hyperbolic sets as the geometric Lorenz attractor. Therefore, it is relevant to study which properties valid for hyperbolic sets are also satisfied by the sectional-hyperbolic sets. \\

Recently in \cite{ww}, it proved that every chain transitive sectional-hyperbolic set with singularities does not have WPOTP.\\

Below we will specify the definitions and results over sectional-hyperbolic dynamics that we will use in this paper.\\

Hereafter $M$ will be a compact manifold  possibly with nonempty boundary endowed  with a Riemannian metric $\langle\cdot,\cdot\rangle$ an induced norm $||\cdot||$. Given $X$ a $C^1$ vector field, inwardly transverse to the boundary (if nonempty) we call $X_t$ its induced {\em flow} over $M$. Define the {\em maximal invariant set} of $X$ as:
$$
M(X)=\displaystyle\bigcap_{t\geq0}X_t(M).
$$
The \textit{orbit} of a point $p \in M(X)$ is defined by $\mathcal{O}(p)=\{X_t(p)\,|\,t\in\mathbb{R}\}$.
A {\em singularity} will be a zero $q$ of $X$, i.e. $X(q)=0$ (or equivalently $\mathcal{O}(q)=\{q\}$)
and a {\em periodic orbit} is an orbit $\mathcal{O}(p)$ such that $X_T(p)=p$ for some minimal $T>0$ and $\mathcal{O}(p)\neq \{p\}$. By a {\em closed orbit}
we mean a singularity or a periodic orbit.\\

Given $p\in M$ we define the {\em omega-limit set},
$\omega_X(p)=\{x\in M\,|\,x=\lim_{n\rightarrow\infty}X_{t_n}(p)$ for some sequence $t_n\rightarrow\infty\}$, if $p \in M(X)$, define the {\em alpha-limit set}
$\alpha_X(p)=\{x\in M:x=\lim_{n\to{\infty}} X_{-t_n}(p),
$ for some sequence $ t_n\to \infty\}$.\\

A compact subset $\Lambda$ of $M$ is called {\em invariant} if
$X_t(\Lambda)=\Lambda$ for all $t\in\mathbb{R}$; {\em transitive} if $\Lambda = \omega_X(p)$ for some $p\in \Lambda$. A compact invariant set $\Lambda$ is {\em attracting} if there is a neighborhood $U$ such that
\[\Lambda=\cap_{t \geq 0}X_t(U),\] 
and is {\em attractor} of $X$, if is an attracting set $\Lambda$  which is {\em transitive}. On the other hand, a compact invariant set $\Lambda$ is {\em Lyapunov stable}, if for all neighborhood $U$ of $\Lambda$, there exists a neighborhood $W$ such that: $X_t(p)\in U$ for every $t\geq 0$ and $p\in W$.

\subsection{Hyperbolic and Sectional-hyperbolic sets}

\begin{defin}
A compact invariant set $\Lambda \subseteq M(X)$ is {\em hyperbolic} if there are positive constants $K,\lambda$and a continuous $DX_t$-invariant splitting of the tangent bundle  $T_{\Lambda}M=E^s_{\Lambda}\oplus E^X_{\Lambda}\oplus E^u_{\Lambda}$,  
 such that for every $x \in \Lambda$ and $t \geq 0$:
\begin{enumerate}
\item [$(1)$] $\| DX_t(x)v^s_x \| \leq K e^{-\lambda t}\| v^s_x \|,\ \ \forall v^s_x \in E^s_x$; 
\item [$(2)$] $ \| DX_t(x)v^u_x \| \geq K^{-1} e^{\lambda t} \| v^u_x \|,\ \ \forall v^u_x \in E^u_x$; 
\item [$(3)$] $E^{X}_{x}=\left\langle X(x)\right\rangle $.
\end{enumerate}
\end{defin}

If $E^s_x\neq 0$ and $E^u_x\neq 0$ for all $x\in \Lambda$ we will say that $\Lambda$
is a {\em saddle-type hyperbolic set}. A closed orbit is hyperbolic, if as a compact invariant set of $X$ is hyperbolic. We will say that a regular point $p$ is hyperbolic if $\omega(p) \cup \Or(P) \cup \alpha(p)$ is a hyperbolic set.\\

The invariant manifold theory \cite{hps} asserts that if $H\subseteq M$ is hyperbolic set of $X$ and $p \in H$, then the topological sets:
$$W^{ss}(p)=\{q \in M: \lim_{t\to\infty} d(X_t(q),X_t(p))=0\}$$
and
$$W^{uu}(p)=\{q \in M(X): \lim_{t\to- \infty} d(X_t(q),X_t(p))=0\}$$
are $C^1$ manifolds in $M$, so called strong stable and unstable manifolds, tangent at $p$ to the subbundles
$E^s_p$ and $E^u_p$ respectively. Saturating them with the flow we obtain the stable and unstable 
manifolds $W^{s}(p)$ and $W^{u}(p)$ respectively, which are invariant. If $p, p' \in H$, we have that 
$W^{ss}(p)$ and $W^{ss}(p')$ are the same or they are disjoint (similarly for $W^{uu}$).\\

\begin{defin}
A compact invariant set $\Lambda \subseteq M(X)$ is {\em sectional-hyperbolic} if every singularity in $\Lambda$ is hyperbolic (as invariant set) and there are a continuous $DX_t$-invariant splitting of the tangent bundle $T_{\Lambda} M = \Df^s_{\Lambda}\oplus \Df^c_{\Lambda}$, and positive constants $K,\lambda$
such that for every $x \in \Lambda$ and $t \geq 0$:

\begin{enumerate}
\item
$\| DX_t(x)v^s_x \| \leq K e^{-\lambda t}\| v^s_x \| ,\ \ \forall v^s_x \in \mathsf{F}^s_x$.
\item
$\| DX_t(x)v^s_x \| \cdot \| v^c_x \|  \leq K e^{-\lambda t} 
\| DX_t(x)v^c_x \| \cdot \| v^s_x \|,\ \ \forall v^s_x \in \mathsf{F}^s_x,\ \ \forall v^c_x \in \mathsf{F}^c_x$.
\item
$ \| DX_t(x)u^c_x , DX_t(x)v^c_x  \|_{X_t(x)} \geq K^{-1} e^{\lambda t}  \| u^c_x , v^c_x \|_x, \ \ 
\forall u^c_x , v^c_x \in \mathsf{F}^c_x$.\\
Where $||\cdot, \cdot ||_x$ it is induced $2$-norm by the Riemannian metrics 
$\langle \cdot, \cdot \rangle_x$ of $T_x\Lambda$, given by
$$||v_x,u_x||_x=\sqrt{\langle v_x,v_x \rangle_x \cdot \langle u_x,u_x\rangle_x - \langle v_x,u_x 
\rangle_x^2}$$ 
for all $x \in \Lambda$ and every $u_x,v_x\in T_x\Lambda$ .
\end{enumerate}
\end{defin}

The third condition guarantees the increase exponential of the area of parallelograms in the central subbundle $\mathsf{F}^c$. Since $ X(x) \in \Df^c_x$ for all $x \in \Lambda$ (see Lemma 4 in \cite {bm}), we will require that the dimension of the central subbundle must be greater than or equal to $2$. In the particular case where the $ dim (\mathsf{F}^c_x) = 2$ we will say that $\Lambda$ is a sectional-hyperbolic set of \textit{codimension $1$} or codimension one sectional-hyperbolic set. \\

Also the invariant manifold theory \cite{hps} asserts that through any point $x$ of a sectional-hyperbolic set $\Lambda$ it passes a 
strong stable manifold $\F^{ss}(x)$, tangent at $x$ to the subbundle $\Df^s_x$, which induces a foliation over a neighborhood of $\Lambda$, that we will call $U_\Lambda$. Saturating them with the flow we obtain the invariant manifold $\F^s(x)$.\\

Unlike hyperbolic sets, the sectional-hyperbolic sets can have regular orbits which accumulate singularities. We have:

\begin{lemma}\label{lema1}
If $\Lambda \subseteq M(X)$ is sectional-hyperbolic set, and $\sigma$ is a singularity in $\Lambda$ then:
$$\F^{ss}(\sigma) \cap \Lambda = \{\sigma\}$$
\end{lemma}

\begin{proof}
See corollary 2 in \cite{bm}.
\end{proof}

Every singularity $\sigma$ in an sectional-hyperbolic set, is hyperbolic, so it's invariant manifolds by the hyperbolic structure, $W^{uu}(\sigma)$ and $W^{ss}(\sigma)$ are well defined. The strong stable manifold  by the sectional-hyperbolic structure $\F^{ss}(\sigma)$, is a submanifold of $W^{ss}(\sigma)$, with respect to your dimension, exists two possibilities:

\begin{enumerate}
\item $dim(W^{ss}(\sigma)) = dim(\F^{ss} (\sigma))$, in this case $W^{ss}(\sigma) = \F^{ss} (\sigma)$.
\item $dim(W^{ss}(\sigma)) = dim(\F^{ss}(\sigma)) + 1$, in this case, we say that the singularity is {\em Lorenz-like}.
\end{enumerate}

Every Lorenz-like singularity $\sigma$, is type-saddle hyperbolic set with at least two negative eigenvalues, one of which is real eigenvalue $\lambda_\sigma$ with multiplicity one such that the real part of the other eigenvalues are outside the closed interval
$[\lambda_\sigma, -\lambda_\sigma]$.\\

We have that a cross section $\Sigma \subset U_\Lambda$ is associated to a Lorenz-like singularity $\sigma$ in a sectional-hyperbolic set $\Lambda$, if $\Sigma$ is very close to $\sigma$, $\Sigma$ intersect one of the stable regular orbits of $\sigma$ associated with the eigenvalue
$\lambda_\sigma$ and $\partial^h \Sigma \cap \Lambda = \emptyset$, \cite{m}, \cite{al}. We denote by $l^*_\Sigma$ the leaf of $\Sigma$ that contains the intersection point. 
\\

Let $\Lambda$ be a sectional-hyperbolic set, $\sigma \in \Lambda$ a singularity Lorenz-like, a {\em singular cross section} associated to $\sigma$, consists of a pair of cross sections $\Sigma_t$ and $\Sigma_b$, 
associated  to $\sigma$ such that $\Sigma_t$ intersects one of 
the stable regular orbits of $\sigma$ associated with the eigenvalue $\lambda_\sigma$ and $\Sigma_b$ intersects the other stable regular orbit.\\

Over a Lorenz-like singularity $\sigma \in \Lambda$, we have $\F^{ss}(\sigma)$ is tangent to the subspace associated the eigenvalues with real part less than $\lambda_\sigma$, and $\F^{ss}(\sigma)$ divide to $W^{ss}(\sigma)$ in two connected components. If $\Lambda$ intersect just one the stable regular orbits of $\sigma$ associated with the eigenvalue $\lambda_\sigma$.

\begin{coro}\label{coroC7.5}
Every point of an attractor sectional-hyperbolic set $\Lambda$ of codimension $1$ with a unique singularity Lorenz-like which is of boundary type, can be approximated by points in $\Lambda$
for which the omega-limit set is a singularity. 
\end{coro}

\begin{proof}
See \cite{y}.
\end{proof}

Another important result about the sectional-hyperbolic set, is the \emph{\textbf{hyperbolic lemma}} (see Lemma 9 in \cite{bm}), which asserts that any invariant compact subset $H$ without singularities of a sectional-hyperbolic set $\Lambda$,  is hyperbolic. In this case, we have that $\Df^s_H=E^s_H$ and $\Df^c_H=E^u_H \oplus E^X_H$, so $W^{ss}(p)=\F^{ss}(p)$ for all $p \in H$.\\

\subsection{Shadowing property}

Before starting to study the shadowing property, it is 
necessary to mention that there are several versions in the literature of the definition of shadowing, in this paper we will present the version of Komuro \cite{ko}, but we modify the definition for the context in that we work; in addition, Komuro present five type of shadowing. 

\begin{defin} \label{definC8.3}
Let $H\subseteq M$ be a compact invariant set, given $\delta,T>0$ a collection $\{x_{i}; t_{i}\}_{0}^{k}$, with $x_{i}\in H$ $t_{i}\geq T$ and $k \in \Dz^+$, 
is called a \textbf{ {\em finite $(\delta, T)$-chain}} of $H$, if $d(X_{t_i}(x_{i}),x_{i+1})\leq \delta$ for every $0\leq i\leq k-1$. For each $t \in [0, \sum_{i=0}^k t_i]$
we denote:
$$x_{0}*t=X_{t-{S_i}}(x_{i})$$ if $S_{i}\leq t<S_{i+1}$ for some $0\leq i\leq k$ where 
$$S_{i}=\begin{cases} 0  & \text{if } i=0 \\ 
\sum_{j=0}^{i-1}t_{j} & \text{if } 0<i\leq k+1 \end{cases}$$
\end{defin}

\begin{defin}\label{definC8.4}
Let $H\subseteq M$ be a compact invariant set, given $\delta,T>0$ a collection $\{x_{i}; t_{i}\}_\Dz$, with $x_{i}\in X$ and $t_{i}\geq T$
is called a \textbf{ {\em $(\delta, T)$-chain}} of $H$, if $d(X_{t_i}(x_{i}),x_{i+1})\leq \delta$ for every $i \in \Dz$. For each $t \in \Dr$ we define
$$x_{0}*t=\varphi_{t-{S_i}}(x_{i})$$ if $S_{i}\leq t<S_{i+1}$ where 
$$S_{i}=\begin{cases} 0  & \text{if } i=0 \\ 
\sum_{j=0}^{i-1}t_{j} & \text{if } i> 0 \\
-\left(\sum_{j=1}^{-i}t_{-j}\right) & \text{if } i< 0
\end{cases}$$
\end{defin}

We consider the following three sets of the reparametrizations:

$$Rep =\{g \in C( \Dr) : g \text{ is a strictly increasing with } g(0)=0\}$$
$$Rep^* =\{g \in Rep: g(\Dr)=\Dr\}$$
$$Rep(\epsilon)=\left\lbrace g\in Rep^*: \left\vert \frac{g(s)-g(t)}{s-t}-1 \right\vert \leq \epsilon \text{ for every } s,t \in \Dr \right\rbrace$$

Using this sets
, we define the concept of $\epsilon$-traced. 

\begin{defin}\label{definC8.5}
Let $H\subseteq M$ be a compact invariant set. A $(\delta, T)$-chain $\{x_{i} ; t_{i}\}_{i\in Z}$ of $H$ is said to be \textbf{ weakly $\epsilon$-traced} (resp. \textbf{ normal $\epsilon$-traced} or \textbf{strongly $\epsilon$-traced}) by a point $x\in H$
if there is a reparametrization $g\in Rep$ (resp. $g\in Rep^*$ or $g \in Rep(\epsilon)$ ) such that
$d(x_{0}*t, X_{g(t)}(x))\leq \epsilon$ for every $t\in \Dr$.
\end{defin}

\begin{defin}\label{definC8.6}
Let $H\subseteq M$ be a compact invariant set. A finite $(\delta, T)$-chain $\{x_{i} ; t_{i}\}_0^k$ of $H$ is said to be \textbf{weakly $\epsilon$-traced} (resp. \textbf{ normal $\epsilon$-traced} or \textbf{ strongly $\epsilon$-traced}) by a point $x\in H$
if there is a reparametrization $g\in Rep$ (resp. $g\in Rep^*$ or $g \in Rep(\epsilon)$) such that
$d(x_{0}*t, X_{g(t)}(x))\leq \epsilon$ for every $t\in \Dr$.
\end{defin}

Finally we establish the definition of the shadowing of Komuro, but we call \textbf{($\cdot$) P.O.T.P (pseudo orbit traced property)}.

\begin{defin}\label{definC8.7}
Let $H\subseteq M$ be a compact invariant set. $H$ has the \textbf{weak POTP} (resp. the \textbf{normal POTP} or the \textbf{strong POTP}) if for any $\epsilon>0$ there are $\delta, T>0$ such that every $(\delta, T)$-chain of 
$H$ can be weakly $\epsilon$-traced (resp. normal $\epsilon$-traced or strongly $\epsilon$-traced) by some point of $X$. In this case we say that $H$ 
have \textbf{WPOTP} (resp. \textbf{NPOTP} or \textbf{SPOTP}).
\end{defin}

\begin{defin}\label{definC8.8}
Let $H\subseteq M$ be a compact invariant set. $H$ has the \textbf{finite POTP} (resp. the \textbf{strong finite POTP}) if for any $\epsilon>0$ there are $\delta, T>0$ such that every finite $(\delta, T)$-chain of 
$H$ can be weakly $\epsilon$-traced (resp. strongly $\epsilon$-traced) by some point of $X$. In this case we say that $H$ 
have \textbf{FPOTP} (resp. \textbf{SFPOTP}).
\end{defin}

Note that we do not define normal finite POTP, this is because in the case of finite $(\delta,T)$-chains, the fact that a reparametrization
is or not surjective does not produce any change, that is, that normal finite POTP is the same that FPOTP.\\

The next Theorem establish the relation between the five definitions of POTP.

\begin{teor} \label{teorC8.2}
Let $H$ a compact invariant set of $X$, then following relations holds: 
\begin{center}
$\begin{array}{c c c c c}
SPOTP &  \Rightarrow & NPOTP & \Rightarrow & WPOTP \\
\Updownarrow & & & & \Downarrow\\
SFPOTP & & \Longrightarrow & & FPOTP 
\end{array}$
\end{center}
\end{teor}

\begin{proof}[Proof]
 See Theorem 7 in \cite{ko}
\end{proof}

In addition, we also have that:

\begin{teor} \label{teorC8.3}
Let $H$ a nonsingular compact invariant set of $X$, then following relations holds: 
\begin{center}
$ SPOTP \Leftrightarrow NPOTP \Leftrightarrow WPOTP \Leftrightarrow SFPOTP \Leftrightarrow FPOTP$ 
\end{center}
\end{teor}

\begin{proof}[Proof]
See Theorem 4 in \cite{ko}
\end{proof}

As a consequence, every hyperbolic attracting set has POTP in any its of five versions. Note that a singularity satisfies definitions trivially as compact invariant set.\\

As we mentioned, in \cite{ww} was proved that if $\Lambda$ is a sectional-hyperbolic chain transitive set with WPOTP, then it admits no singularity.\\

We observe that if an invariant compact set does not have FPOTP, then it does not have any of the shadowing definitions presented in this section. This, together with the following Lemma, will be the tools that we will use in the next section, to prove that some kind attractor sectional-hyperbolic sets doesn't have shadowing.\\

\begin{lemma}\label{lemaC8.1}
Let $H$ a compact invariant set of $X$. The following statements are equivalent:
\begin{enumerate}
\item $H$ has FPOTP (resp. SFPOTP),
\item for every $\epsilon>0$ and $T>0$ there is $\delta>0$ such that every finite $(\delta,T)$-chain is weakly $\epsilon$-traced (Resp. strong $\epsilon$-traced).
\end{enumerate}
\end{lemma}

\begin{proof}[Proof]
See Lemma 3.2 in \cite{ko}
\end{proof}

As the main result of this paper, we prove the following Theorem:

\begin{teor}
If $\Lambda$ is a codimension one sectional-hyperbolic attractor set with a unique singularity Lorenz-like, which is of boundary-type, then $\Lambda$ does not have FPOTP.
\end{teor}

\begin{coro}
If $\Lambda$ is a codimension one sectional-hyperbolic attractor set with a unique singularity Lorenz-like, which is of boundary-type, then $\Lambda$ does not have shadowing.
\end{coro}

\section{Side Points}

In the case of the sectional-hyperbolic set of codimension $1$, over any regular point in it, the local stable manifold divides the neighborhoods of the point, in two connected components, we will use this, to we define a side point and study the shadowing property (specifically FPOTP) in sectional-hyperbolic sets.\\

Given $x \in \Lambda$ a sectional-hyperbolic set, we denote $\F^s_\epsilon(x)$ the connected component of $B_\epsilon(x) \cap \F^{s}(x)$ that contains $x$, where $B_\epsilon(x)$ denote the ball of the radius $\epsilon$ centered in $x$.

\begin{defin}\label{definC8.9}
Let $\Lambda$ a sectional hyperbolic set of codimension $1$, 
we say that $x\in \Lambda^*$ is a {\em side point}, if there exists $\epsilon >0$ 
such that $\Lambda$ intersects only one of the connected components of $B_\epsilon(x)\setminus \F^s_\epsilon(x)$, and  
we say that $y \in \Lambda^*$ is a {\em bi-side point}, if for all $\epsilon >0$, 
$\Lambda$ intersects booth connected components of $B_\epsilon(x)\setminus \F^s_\epsilon(x)$. 
\end{defin}

We can observe that there is a case that we did not mention in the previous definition, when $\Lambda$ does not intersect
any of the connected components of $B_\epsilon(x)\setminus\F^s_\epsilon(x)$;
although it is not so common, it can happen. The figure 1 in \cite{bm4} shows a sectional-hyperbolic set, composed by two singularities, one Lorenz-like and other not Lorenz-like, and a heteroclinical orbit between them; where, the regular points are neither side or bi-side. This example is interesting, due to is also an example of a sectional-hyperbolic set has FPOTP.\\

Now, we denote:
$$Sd(\Lambda)=\{x\in \Lambda : x \text{ is a side point}\}$$
and
$$\text{2-}Sd(\Lambda)=\{x\in \Lambda : x \text{ is a bi-side point}\}$$

\begin{lemma}\label{lemaC8.2}
Let $\Lambda$ a sectional hyperbolic Lyapunov stable set  of codimension $1$, then every hyperbolic point, is in 2-$Sd(\Lambda)$ 
\end{lemma}

\begin{proof}[Proof]
It is sufficient observing that $W^{uu}(p)$ is composed of two unstable branches, each of them, in a different component of the
$B_\epsilon(x)\setminus\F^s_\epsilon$, and since $\Lambda$ is Lyapunov stable set contains these branches. 
\end{proof}

Note that the continuous dependence of flow, guaranteed that, 2-$Sd(\Lambda)$ is positively invariant, but we can't claim the same 
for $Sd(\Lambda)$ in general, we prove that
is true in for some attractor sets with FPOTP, for this we introduce the following definitions and denotations.\\

If $x \in Sd(\Lambda)$ we denote by $B^+_{\epsilon}(x)$ the connected component of $B_{\epsilon_x}(x)\setminus \F^s_{\epsilon_x}(x)$ that intersect $\Lambda$ and 
$B^-_{\epsilon}(x)$ the other component.\\

Given $\sigma$ a Lorenz-like singularity in $\Lambda$ a sectional-hyperbolic set of codimension $1$ and $\R'_\sigma =\{\Sigma_t,\Sigma_b\}$ a singular cross section associated with $\sigma$, of the time $T_\R$.
We have that $\sigma$ divides 
to $W^{uu}(\sigma)$ in two connected components that we call $W^l$ and $W^r$. Something similar happens with $l^*_{\Sigma_b}$ that divides $\Sigma_b$, and  $l^*_{\Sigma_t}$
that divides $\Sigma_t$; we denote these connected components with $\{l,r\}$ of corresponding way, that is, $\Sigma_b^l$ is the component of $\Sigma_b$ that when leaving 
$\Sigma_b$ begins to accumulate $W^l$, and $\Sigma_b^r$ is the component of $\Sigma_b$ that when leaving 
$\Sigma_b$ begins to accumulate $W^r$  \cite{bm3},\cite{amlb}. In the same a way, we denote $\Sigma_t^l$ and $\Sigma_t^r$.\\

We denote $\R_\sigma^l=\Sigma_b^l \cup \Sigma_t^l$ and $\R^r_\sigma=\Sigma_b^r \cup \Sigma_t^r$. Given $\delta>0$ we denote 
$V_\delta^l(l^*_{\R_\sigma}) = V_\delta(l^*_{\R_\sigma}) \cap \R^l_\sigma$ and 
$V_\delta^r(l^*_{\R_\sigma}) = V_\delta(l^*_{\R_\sigma}) \cap \R^r_\sigma$, where
$V_\delta(l^*_{\R_\sigma})$ is a subset of $R_\sigma$ such that $x \in V_\delta(l^*_{\R_\sigma})$ if only if there is $y \in l^*_{\R_\sigma}$ with $d(x,y)<\delta$ and the leaf of $\R_\sigma$ that contains $x$ is contained in $V_\delta(l^*_{\R_\sigma})$.\\

Let $\gamma>0$ we define:\\ 
$$W^l_\gamma=\{x\in W^l : \mathcal{O}^-(x) \subseteq B_\gamma(\sigma)\} \, \text{ and } \, W^r_\gamma=\{x\in W^r : \mathcal{O}^-(x) \subseteq B_\gamma(\sigma)\}$$

Since $\sigma$ is hyperbolic, by Hartman-Grobman Theorem, 
there exists $\beta_\sigma >0$ 
small enough such that $W^r_{2\beta_\sigma}$ 
and $W^l_{2\beta_\sigma}$ are open curves $(\sigma, y^r_\sigma)$
and $(\sigma, y^l_\sigma)$ respectively of size $2\beta_\sigma$, with $y^r_\sigma \in W^r$ and $y^l_\sigma \in W^l$, satisfying:

\begin{align}\label{ecuaC8.1}
\hspace{1cm} B_{\beta_\sigma}(y^\star_\sigma)\cap \mathcal{O}^+(l^*_\R)=\emptyset \,\,
\text{ and } \,\, B_{\beta_\sigma}(y^\star_\sigma)\cap \R =\emptyset.
\end{align}
whit $\star \in \{l,r\}$. Now, using the continuity of the flow, 
we can choose $\epsilon_\sigma>0$ and $T_\sigma$ such that:

\begin{equation}\label{ecuaC8.2}
 X_{(0,{T_\sigma})}(x)\cap B_{\beta_\sigma}(y^l_\sigma) \neq \emptyset \, \text{ and } \,
X_{(0,{T_\sigma})}(x)\cap B_{\beta_\sigma}(y^r_\sigma) = \emptyset \, 
\end{equation}

for all $x\in  X_{[-T_\R,T_\R]}(V^l_{\epsilon_\sigma}(l^*_\R))$; and

\begin{equation}\label{ecuaC8.3}
X_{(0,{T_\sigma})}(x)\cap B_{\beta_\sigma}(y^r_\sigma) \neq \emptyset \, \text{ and } \,
X_{(0,{T_\sigma})}(x)\cap B_{\beta_\sigma}(y^l_\sigma) = \emptyset \, 
\end{equation}

for all $x\in  X_{[-T_\R,T_\R]}(V^r_{\epsilon_\sigma}(l^*_\R))$.

\begin{teor}\label{teorC8.4}
Let $\Lambda$ a codimension one sectional hyperbolic Lyapunov stable set. If $p \in Sd(\Sigma))$ then $\omega(p)$ not is a closed orbit.
\end{teor}

\begin{proof}[Proof]
Suppose that there is $p \in Sd(\Lambda)$ such that $\omega(p) =\{\sigma\}$, with $\sigma \in Sing(X)\cap \Lambda$. We choose $\R$ a cross section associated with $\sigma$
of the time $T_\R$, such that the positive orbit of $q$ intersect to $\R$ in $\l^*_R$.\\ 

Since $\Lambda$ is a Lyapunov stable $W^{uu}(\sigma) \subseteq \Lambda$. We choose $\gamma >0$ such that $X_{[-T_\R,T_\R]}(\R) \subseteq B_\gamma(\R)$. On the other hand, 
we have that there exists $T_p>0$ such that $X_{T_p}(p) \in \l^*_\R$.\\

We define:
$$4\epsilon=\min \left\lbrace \epsilon_p,\gamma,\epsilon_\sigma, \beta_\sigma \right\rbrace,$$
and 
$$2T=\min\left\lbrace T_p ,T_\sigma \right\rbrace.$$

By hypothesis $\Lambda$ has FPOTP, so using the Lemma \ref{lemaC8.1}, by $\epsilon$ and $T$ there is $\delta>0$ such that every finite $(\delta,T)$-chain is $\epsilon$-traced, without 
loss of generality, we can assume that $\delta < \frac{\epsilon}{2}$.\\

We will build a $(\delta,T)$-chain as follows.\\

We consider $x_0=p$, by the tubular flow Theorem $X_{T_p}(B^+_\epsilon(p))$ 
intersect continuously only one of the components of
$B_\epsilon(X_{T_p}(p)) \setminus \F^s_\epsilon(X_{T_p}(p))$ that we call $B^*_\epsilon(X_{T_p}(p))$,  without loss of generality, 
we can assume that this component intersect $V^r_\epsilon (l^*_{\R})$ but not $V^l_\epsilon (l^*_{\R})$.\\ 

We take $t_0>T$ such that $X_{t_0}(x_0) \in B_{\frac{\delta}{2}}(\sigma)$, and finally, we choose  $x_1 \in W^l_{\beta_{\sigma}} \cap B_{\frac{\delta}{2}}(\sigma)$ 
and $t_1>T$ such that $X_{t_1}(x_1)=y^l_\sigma$.\\

We have that $\{x_i,t_i\}_0^1$ is a finite $(\delta,T)$-chain 
then exists $z \in \Lambda$ y $g\in Rep$ 
such that $$d(x_0*t,X_{g(t)}(z))<\epsilon, \text{ for every } t\in [0,t_0+t_1]$$

Taking $t=0$ we have that: $z \in B_\epsilon(p)$, so $z \in \F^s_\epsilon(p) \text{ or } z\in B^+_\epsilon(p)$, but if $z\in \F^s_\epsilon(p)$, then taking
$t=t_0+t_1$, we have that $X_{g(t_0+t_1)}(z) \in B_{\beta_\sigma}(y^l_\sigma)\cap \mathcal{O}^+(l^*_\R)$ and this is a contradiction of \ref{ecuaC8.1}. Therefore, 

\begin{align}\label{ecuaC8.4}
\hspace{1cm} z \in B^+_\epsilon(p)
\end{align}
Now, taking $t=T_p$, we have that $X_{g(T_p)} \in B_\epsilon(X_{T_p}(p))$, since the flow is injective, $X_{[0,g(T_p)]}(z)$ can not intersect to $X_{[0,T_p]}(\F^{ss}(p)) \cup 
\F^s_\epsilon(p) \cup \F^s_\epsilon(X_{T_p}(p))$, 
so by \ref{ecuaC8.4}: 
$$X_{g(T_p)}(z) \in B^*_\epsilon(X_{T_p}(p))$$
Now, since  $B^*_{\epsilon}(X_{T_p}(p))$,  intersect $V^r_\epsilon (l^*_{\R_{\sigma}})$ but not $V^l_\epsilon (l^*_{\R_{\sigma}})$, $\epsilon< \epsilon_{\sigma}$ and 
$\epsilon < \gamma$, we have that:
$$X_{g(T_p)}(z) \in X_{[-T_\R,T_\R]}(V_\epsilon^r(l^*_{\R}))$$
so by \ref{ecuaC8.3}
\begin{align}\label{ecuaC8.5}
\hspace{1cm}X_{(0,T_{\sigma})}(X_{g(T_p)}(z)))\cap  B_{\beta_\sigma}(y^r_\sigma) \neq \emptyset \text{, } X_{(0,T_{\sigma})}(X_{g(T_p)}(z)))\cap  B_{\beta_\sigma}(y^l_\sigma) = \emptyset
\end{align}
On the other hand, taking $t= t_0+t_1$, $X_{g(t_0+t_1)}(z) \in B_\epsilon(y^l_\sigma)$, and since $\beta_{\sigma}>\epsilon$, by \ref{ecuaC8.5}:
$$g(t_0+t_1)>T_{\sigma} + g(T_p)$$
so for every  $r \in [0, T_{\sigma}]$ we have that $g^{-1}(r+g(T_p))$ is well defined. Therefore by \ref{ecuaC8.5}, there is $T \in (0,T_{\sigma})$, such that:
$$X_{g(g^{-1}(T +g(T_p))}(z) = X_{T +g(T_p)}(z) \in B_{\epsilon_{\sigma}}(y^r_\sigma).$$ 
with $T_p < g^{-1}(T +g(T_p))<t_0+t_1$, then 
$$d(X_{g(g^{-1}(T +g(T_p))}(z),x*(T_p+T))<\epsilon < \frac{\beta_{\sigma}}{4}$$
so
$$d(y^r_\sigma,x*(T_p+T))<\frac{3\beta_{\sigma}}{4}$$
we know that 
$$ x*(T_p+T) \in  \Or^+(l^*_{\sigma}) \cup (\sigma,y^l_\sigma)$$
but by definition of $(\sigma,y^r_\sigma)$, $(\sigma,y^l_\sigma)$ and \ref{ecuaC8.1}
$$(\Or^+(l^*_{\sigma}) \cup (\sigma,y^l_\sigma))\cap B_{\beta_\sigma}(y^r_\sigma) =\emptyset.$$
This is a contradiction and this case of theorem follows.\\

The proof in the case that $\omega(p)=\Or(q)$ with $q$ a periodic point, is analogous to the previous case, due to, it has similar conditions, such as, the existence and dimension of the strong unstable manifold by the
hyperbolic structure, and linearized local behavior. Since $p$ is a side point, when accumulates $\Or(q)$ the points in $B^+_\epsilon(p)$ only can approximate just one of the  branches of $W^ {uu} (q) $ without intersect $\F^s (q) =\F^s (p)$, so a $(\delta,T)$-chain building with the orbit of $p$ and the branch 
of the stronger unstable manifold that is not accumulated, does not cant shadowed; and we obtain the result. 
\end{proof}

\section{Proof of the Main Theorem}

\begin{teor}\label{teorC8.6}
Let $\Lambda$ a codimension one transitive sectional hyperbolic Lyapunov stable set, such that every point in $\Lambda$ can be approximated by points of $\Lambda$ for which the
omega limit set is a singularity. If $\Lambda$ has FPOTP, then $X_t(Sd(\Sigma)) \subseteq Sd(\Sigma)$ for all $t \geq 0$ and 
$$\left( \bigcup_{p \in Sd(\Lambda)} \F^{ss}(p) \right)\cap \left( \bigcup_{q \in \text{2-}Sd(\Lambda)}\F^{ss}(q)\right)=\emptyset $$

\end{teor}

\begin{proof}[Proof]
First we prove that $Sd(\Lambda)$ is a invariat set.\\

If $\Lambda$ has no singularities, then by hyperbolic Lemma is hyperbolic set and by Lemma \ref{lemaC8.2}, $Sd(\Lambda)= \emptyset$ and the result is trivially.\\

We assume that $\Lambda$ has a singularities. we have that $W^u(\sigma) \subseteq \Lambda$ for all $\sigma \in Sing(X) \cap \Lambda$, due to, $\Lambda$ is Lyapunov stable.\\ 

Suppose that exists $p \in Sd(\Lambda)$ and $t>0$ such that $X_t(p) \notin Sd(\Lambda)$. Since $p\in Sd(\Sigma)$, then exists $\epsilon_p>0$ such that 
$B^-_{\epsilon_p}(p)\cap \Lambda = \emptyset$. By Theorem \ref{teorC8.4}, $p \neq \{\sigma\}$ for all $\sigma \in Sing(x)$.\\

We consider for every singularity $\sigma \in \Lambda$, $\R'_\sigma$ a singular cross section associated with $\sigma$ far enough from $p$, we take 
$\R=\bigcup_{\sigma \in Sing(X) \cap \Lambda} \R_\sigma$; if $\R$ is the time $T_\R$, we choose $\gamma >0$ such that $X_{[-T_\R,T_\R]}(\R) \subseteq B_\gamma(\R)$.\\  

We define:
$$4\epsilon=\min \left( \left\lbrace \epsilon_p,\gamma \right\rbrace \cup \left( \bigcup_{\sigma \in Sing(x) \cap \Lambda} \{\epsilon_\sigma, \beta_\sigma\} \right)
\right),$$
and 
$$2T=\min
\left( \bigcup_{\sigma \in Sing(x) \cap \Lambda} \{T_\sigma\} \right)
.$$

By hypothesis $\Lambda$ has FPOTP, so using the Lemma \ref{lemaC8.1} by $\epsilon$ and $T$ there is $\delta>0$ such that every finite $(\delta,T)$-chain is $\epsilon$-traced, without loss of generality, 
we can assume that $\delta < \frac{\epsilon}{2}$.\\

We will build a $(\delta,T)$-chain as follows.\\

We consider $x_0=p$. Since, $p \notin W^s(Sing(X)$ and there exists $t>0$ with $X_t(p) \in \text{2-}Sd(\Sigma)$ and this set is positively invariant, then there is 
$T_p>T$ such that $q=X_{T_p}(p) \in \text{2-}Sd(\Lambda) \setminus B_{\epsilon_p}(p)$ and $q$ far enough $\R$, we choose $t_0=T_p>T$. By the tubular flow Theorem $X_{t_0}(B^+_\epsilon(p))$ 
intersect continuously only one of the components of
$B_\epsilon(q) \setminus \F^s_\epsilon(q)$ that we call $B^*_\epsilon(q)$.\\ 

Since that $q \in \text{2-}Sd(\Lambda)$ by hypothesis, we have that 
in both connected component of $B_\epsilon(q) \setminus \F^s_\epsilon(q)$ exits points in $\Lambda$ for which the omega limit is a singularity, choose $x_1$ one of these
points such that $x_1 \notin B^*_\epsilon(q)$ and $x_1 \in B_\delta(q)$. We call $B^*_\epsilon(x_1)$ the connected componet of $B_\epsilon(x_1)\setminus \F^s_\epsilon(x)$
that
contained $q$.\\

On the other hand, as $\omega(x_1) =\{\sigma\}$ for some $\sigma \in \Lambda\cap Sing(X)$, and $q$ is far from $\R$, there exists $s>0$ such that 
$X_s(x_1) \in l^*_{\R_{\sigma}}$. Using again the tubular flowh Theorem, $X_{s}(B^*_{\epsilon}(x_1))$ 
intersect continuously only one of the components of
$B_\epsilon(X_s(x_1)) \setminus \F^s_\epsilon(X_s(x_1))$, we call $B^*_{\epsilon}(X_s(x_1))$,  without loss of generality, 
we can assume that this component intersect $V^r_\epsilon (l^*_{\R_{\sigma}})$ but not $V^l_\epsilon (l^*_{\R_{\sigma}})$.\\ 

We take $t_1>T$ such that $X_{t_1}(x_1) \in B_{\frac{\delta}{2}}(\sigma)$, and finally, we choose  $x_2 \in W^r_{\beta_{\sigma}} \cap B_{\frac{\delta}{2}}(\sigma)$ 
and $t_2>T$ such that $X_{t_2}(x_2)=y^l_\sigma$.\\

We have that $\{x_i,t_i\}_0^2$ is a finite $(\delta,T)$-chain 
then exists $z \in \Lambda$ y $g\in Rep$ 
such that $$d(x_0*t,X_{g(t)}(z))<\epsilon, \text{ for every } t\in [0,t_0+t_1+t_2]$$
Next we will track the places through which the orbit of $z$ can pass.\\

Taking $t=0$ we have that: $z \in B_\epsilon(p)$, so 
\begin{align}\label{ecuaC8.6}
\hspace{1cm} z \in \F^s_\epsilon(p) \text{ or } z\in B^+_\epsilon(p)
\end{align}
Taking $t=t_0$, we have that $X_{g(t_0)} \in B_\epsilon(q)$, since 
that the flow is injective, $X_{[0,g(t_0)]}(z)$ can not intersect of a transverse way to $X_{[0,t_0]}(\F^{ss}(p)) \cup \F^s_\epsilon(p) \cup \F^s_\epsilon(q)$, 
then by \ref{ecuaC8.6} $X_{g(t_0)}(z) \in \F^s_\epsilon(q)$  or $X_{g(t_0)}(z)\in B^+_\epsilon(q)$ and as $X_1$ is in the connecte component of 
$B_\epsilon(q) \setminus \F^s_\epsilon(q)$ different of $B^*_\epsilon(q)$ and $d(x_1, q) \leq \delta < \frac{\epsilon}{2}$, we conclude that:
\begin{align}\label{ecuaC8.7}
\hspace{1cm} X_{g(t_0)}(z) \in B^*_\epsilon(x_1)
\end{align}
Taking $t=t_0+s$, $X_{g(t_0+s)} \in B_\epsilon(X_s(x_1))$, again as the flow is injective, $X_{[g(t_0),g(t_0+s)]}(z)$ can not intersect to $X_{[0,s]}(\F^{ss}(x_1)) \cup 
\F^s_\epsilon(x_1) \cup \F^s_\epsilon(X_s(x_1))$, 
so by \ref{ecuaC8.7}: 
$$X_{g(t_0+s)}(z) \in B^*_\epsilon(X_s(x_1))$$
Now, since  $B^*_{\epsilon}(X_s(x_1))$,  intersect $V^r_\epsilon (l^*_{\R_{\sigma}})$ but not $V^l_\epsilon (l^*_{\R_{\sigma}})$, $\epsilon< \epsilon_{\sigma}$ and 
$\epsilon < \gamma$, we have that:
$$X_{g(t_0+s)}(z)) \in X_{[-T_\R,T_\R]}(V_\epsilon^r(l^*_{\sigma}))\subseteq X_{[-T_{\sigma},T_{\sigma}]}(V_\epsilon^r(l^*_{\sigma'}))$$
so by \ref{ecuaC8.3}
\begin{align}\label{ecuaC8.8}
\hspace{0,5cm}{(0,T_{\sigma})}(X_{g(t_0+s)}(z)))\cap  B_{\beta_{\sigma'}}(y^r_\sigma) \neq \emptyset \text{, } X_{(0,T_{\sigma})}(X_{g(t_0+s)}(z)))\cap  B_{\beta_{\sigma}}(y^l_\sigma)= 
\emptyset
\end{align}
Taking $t= t_0+t_1+t_2$, $X_{g(t_0+t_1+t_2)}(z) \in B_\epsilon(y^l_\sigma)$, and since $\beta_{\sigma}>\epsilon$, by \ref{ecuaC8.8}:
$$g(t_0+t_1+t_2)>T_{\sigma} + g(t_0+s)$$
so for every  $r \in [0, T_{\sigma}]$ we have that $g^{-1}(r+g(t_0+s))$ is well defined. Therefore, by \ref{ecuaC8.8}, there is $T \in (0,T_{\sigma}]$, such that:
$$X_{g(g^{-1}(T +g(t_0+s))}(z) = X_{T +g(t_0+s)}(z) \in B_{\epsilon_{\sigma}}(y^r_\sigma),$$ 
with $t_0+s < g^{-1}(T +g(t_0+s))<t_0+t_1+t_2$, then 
$$d(X_{g(g^{-1}(T +g(t_0+s))}(z),x*(t_0+s+T))<\epsilon < \frac{\beta_{\sigma}}{4}$$
so
$$d(y^r_\sigma,x*(t_0+s+T))<\frac{3\beta_{\sigma}}{4}$$
we know that 
$$ x*(t_0+s+T) \in  \Or^+(l^*_{\sigma'} \cup (\sigma,y^l_\sigma)$$
but by definition of $(\sigma,y^r_\sigma)$, $(\sigma,y^l_ \sigma)$ and \ref{ecuaC8.1}
$$(\Or^+(l^*_{\sigma}) \cup (\sigma,y^l_\sigma))\cap  B_{\beta_\sigma}(y^l_\sigma) =\emptyset$$
This is a contradiction , so $Sd(\Lambda)$ is an invariant set.\\




Now, suppose that exists $p \in Sd(\Sigma))$ and $q \in \text{2-}Sd(\Sigma)$, such that $$q \in \F^{ss}(p).$$ Let $\epsilon$, $T$ and $\delta$ as in the previous without consider $\epsilon_p$. We take $x_0=p$, since $q \in \F^{ss}(p)$ then $\F^{ss}(q) = \F^{ss}(p)$ and by Lemma \ref{lemaC8.2}, $q \neq \{\sigma\}$ for all $\sigma \in Sing(X)$ and since $\F^{ss}_\Lambda$ is positively invariant and contracting, then there is $t_0 > T$ such that $d(X_{t_0}(p),X_{t_0}(p))< \frac{\delta}{4}$, $X_{t_0}(q) \in \F^s_\epsilon(X_{t_0}(p))$ and far enough from $\R$.\\

Since that $\text{2-}Sd(\Sigma)$ is positively invariant, then $X_{t_0}(q) \in \text{2-}Sd(\Lambda)$, so, we have that in both connected component of $B_{\frac{\delta}{4}}(X_{t_0}(q)) \setminus \F^s_{\frac{\delta}{4}}(X_{t_0}(q))$ there are points in $\Lambda$ for which the omega limit is a singularity. Then, we proceed equal to the first part, by choosing $x_1$, $x_2$, $t_1$ and $t_2$ and we show that the $(\delta,T)$-chain, $\{x_i,t_i\}^2_0$, not is $\epsilon$-traced. The result follows.




\end{proof}

\begin{proof}[Proof of the Main Theorem]
Let be $\sigma$ the unique singularity of $\Lambda$ and $z\in \Lambda$ such that $\omega(z)=\Lambda$. By theorem 1.1 in \cite{aml}, $\Lambda$ has a periodic point $p$. Using the orbit of $z$
we have that for every $\epsilon>0$ there is a
trajectory from a point $\epsilon$-close to $p$ to a point $\epsilon$-close to $\sigma$, therefore $\Lambda$ satisfies the conditions of the theorem 10 in \cite{bss}, so
there exists $w \in \Lambda$ such that $\alpha(w) = \alpha(p)$ and $\omega(p)$ is a singularity, and this case, necessarily, $\omega(w)=\{\sigma\}$.\\ 

Now, since the $\sigma$ is of boundary type, reasoning as in case 2 of Theorem 2.2 in \cite{bs}, we have that there exists $\Sigma$ a singular cross section associated with $\sigma$, such that every leaf of $F^s_\Sigma$ have a point in 2-$Sd(\Lambda)$.\\

Suppose that exists a point in $q \in Sd(\Lambda)$, and $\Lambda$ has FPOTP, by corollary \ref{coroC7.5} and Theorem \ref{teorC8.6}, we have that
$X_t(q)\in Sd(\Lambda)$ for all $t>0$.\\ 

On the other hand, we consider $\omega(q)$. If $ \sigma \in \omega(q)$, then the positive orbit of $q$ intersect $\Sigma$, then there is $T>0$ and $y \in \text{2-}Sd(\Lambda)$,
such that
\begin{align}\label{ecuaC8.9}
\hspace{1cm} x = X_T(q) \in \F^{ss}(y)
\end{align}

but it contradicts the Theorem  \ref{teorC8.6}, since $x \in Sd(\Lambda)$.\\

Now, if $\sigma \notin \omega(q)$ then $\omega(q)$ is hyperbolic set, then we take $s \in \omega(q)$ and we can build a cross section $\Sigma_2$, by using the $W^{uu}(s)$ and the strong stable manifold of the points in it. So, every leaf in $\Sigma_2$ has a point in 2-$Sd(\Lambda)$, when the positive orbit of $q$ accumulate $s$, intersect $\Sigma_2$. 
Therefore, we conclude again (\ref{ecuaC8.9}), and we obtain a contradiction. Thus, if $Sd(\Lambda)\neq\emptyset$, then $\Lambda$ does has not FPOTP.\\ 

Finally, we observe that since $\Lambda$ is an attractor and $\sigma$ is of boundary type, we have that for $\gamma>0$ small enough, the points on $W^l_\gamma$ and $W^r_\gamma$ are side points, and therefore $\Lambda$ does has not FPOTP.
\end{proof}

\end{document}